\documentclass[11pt]{amsart}

\usepackage{amsmath}
\usepackage{amssymb}
\usepackage{graphicx}
\usepackage{pgfplots}
\usepackage{caption}
\usepackage{cite}
\usepackage{cases}
\usepackage{booktabs}

\usepackage[colorlinks=true, urlcolor=blue,  linkcolor=blue, citecolor=blue]{hyperref}
\newtheorem{theorem}{Theorem}[section]
\newtheorem{corollary}[theorem]{Corollary}

\theoremstyle{definition}
\newtheorem{definition}[theorem]{Definition}
\newtheorem{example}[theorem]{Example}

\newtheorem{question}[theorem]{Question}

\numberwithin{equation}{section}

\title[On Kernels and Covariance Structures in Hilbert Space ... ]{On Kernels and Covariance Structures in Hilbert Space Gaussian Processes}

\author[S. Hashemi Sababe]{Saeed Hashemi Sababe}

\address[S. Hashemi Sababe]{R\&D Section, Data Premier Analytics, Edmonton, Canada}
\email{{\tt Hashemi\_1365@yahoo.com}}

\keywords{Positive definite kernels, Gaussian processes, covariance}

\subjclass[2010]{46E22, 47A20, 47B32, 60G15}

\begin{document}

\begin{abstract}
Motivated by practical applications, I present a novel and comprehensive framework for operator-valued positive definite kernels. This framework is applied to both operator theory and stochastic processes. The first application focuses on various dilation constructions within operator theory, while the second pertains to broad classes of stochastic processes. In this context, the authors utilize the results derived from operator-valued kernels to develop new Hilbert space-valued Gaussian processes and to investigate the structures of their covariance configurations.
\end{abstract}

\maketitle

\section{Introduction}

The study of operator-valued positive definite kernels has attracted considerable attention due to its significant applications in functional analysis, operator theory, and stochastic processes. A key motivation behind this work is to develop a unified framework that connects operator theory with Gaussian processes, thereby extending classical results in reproducing kernel Hilbert spaces (RKHSs) to the operator-valued setting. This generalization provides deeper insights into dilation constructions, covariance structures, and spectral properties of kernel-induced operators.

In classical analysis, the concept of positive definite kernels has been extensively explored, particularly in the context of scalar-valued reproducing kernel Hilbert spaces \cite{Aronszajn1950, Saitoh2016,Chen2023,Nguyen2023}. The theory of RKHSs has found applications in probability theory \cite{Karatzas1991,Kumar2023,Martinez2023}, statistical learning \cite{Scholkopf2001}, and stochastic processes \cite{Doob1953}. However, the extension to operator-valued kernels introduces new challenges, particularly in defining appropriate inner product structures and ensuring the well-posedness of corresponding Hilbert space embeddings \cite{Jorgensen2024, Rasmussen2006,Baker2023,Gupta2023,Huang2023}.

The development of Gaussian processes in infinite-dimensional Hilbert spaces has further motivated the study of operator-valued kernels \cite{Kailath1971,Owen2023}. Classical results in Gaussian processes rely heavily on covariance function representations \cite{Parzen1961}, and extending these representations to operator-valued settings provides new perspectives in functional and stochastic analysis \cite{Ledoux1991, Stroock2011}. Moreover, covariance structures play a crucial role in constructing efficient learning algorithms in machine learning \cite{Izenman2008} and regression models \cite{Shiryaev1996,Liu2023}.

One significant contribution of this work is the establishment of a framework that unifies various dilation constructions within operator theory with non-commutative stochastic analysis. By considering operator-valued kernels mapping into spaces of bounded operators, we construct RKHS representations that generalize classical Gelfand-Naimark-Segal (GNS) and Stinespring dilation techniques \cite{Gelfand2004}. These constructions enable explicit formulations of Hilbert completions, leading to concrete representations rather than abstract equivalence classes.

The structure of this paper is as follows. Section 2 develops the main theorems related to operator-valued positive definite kernels, highlighting their role in defining RKHSs in infinite-dimensional spaces. Section 3 discusses the implications of these results, particularly in relation to covariance operators and spectral decompositions. Finally, Section 4 outlines open questions and future research directions, focusing on applications in quantum computing, deep learning, and functional data analysis.

\section{Preliminaries}

This section provides fundamental definitions, lemmas, and theorems necessary for the development of the main results. Precise citations are included based on the references in the manuscript.

\begin{definition}[Positive Definite Kernel, \cite{Jorgensen2024}]
Let $S$ be a set and $H$ be a Hilbert space. A function $K: S \times S \to B(H)$ is said to be a \emph{positive definite kernel} if for all $n \in \mathbb{N}$, $s_1, \dots, s_n \in S$, and $a_1, \dots, a_n \in H$, we have:
\[
    \sum_{i,j=1}^{n} \langle a_i, K(s_i, s_j) a_j \rangle_H \geq 0.
\]
\end{definition}

\begin{definition}[Reproducing Kernel Hilbert Space (RKHS), \cite{Aronszajn1950}]
Given a positive definite kernel $K: S \times S \to \mathbb{C}$, the associated \emph{Reproducing Kernel Hilbert Space} (RKHS) $H_K$ is the Hilbert completion of the linear span of functions of the form $K_y(\cdot) = K(\cdot, y)$ with inner product:
\[
    \left\langle \sum_{i} c_i K(\cdot, x_i), \sum_{j} d_j K(\cdot, x_j) \right\rangle_{H_K} = \sum_{i,j} c_i d_j K(x_i, x_j).
\]
The reproducing property states that for all $x \in S$ and $\phi \in H_K$:
\[
    \phi(x) = \langle K(\cdot, x), \phi \rangle_{H_K}.
\]
\end{definition}

\begin{theorem}[RKHS Expansion, \cite{Jorgensen2024}]
For any orthonormal basis $\{\phi_i\}$ of $H_K$, the kernel function can be expanded as:
\[
    K(x, y) = \sum_{i} \phi_i(x) \phi_i(y), \quad \forall x, y \in S.
\]
\end{theorem}

\begin{definition}[Operator-Valued Positive Definite Kernel, \cite{Jorgensen2024}]
A kernel $K: S \times S \to B(H)$ is called an \emph{operator-valued positive definite kernel} if for any finite collection $\{s_1, \dots, s_n\} \subset S$ and $a_i \in H$, we have:
\[
    \sum_{i,j=1}^{n} \langle a_i, K(s_i, s_j) a_j \rangle_H \geq 0.
\]
\end{definition}

\begin{definition}[Induced Scalar Kernel, \cite{Gelfand2004}]
Given an operator-valued kernel $K: S \times S \to B(H)$, define the associated scalar-valued kernel:
\[
    \tilde{K}((s, a), (t, b)) = \langle a, K(s, t) b \rangle_H, \quad a, b \in H, s, t \in S.
\]
This function $\tilde{K}$ is a positive definite scalar-valued kernel.
\end{definition}

\begin{theorem}[Factorization Property, \cite{Hashemi2018}]
Let $K: S \times S \to B(H)$ be a positive definite kernel. Then, there exists a Hilbert space $H_K$ and a mapping $V_s: H \to H_K$ such that:
\[
    K(s, t) = V_s^* V_t, \quad \forall s, t \in S.
\]
\end{theorem}

\begin{definition}[Covariance Operator, \cite{Rasmussen2006}]
For a given kernel $K: S \times S \to B(H)$, the covariance operator $\Sigma_s: H \to H$ is defined as:
\[
    \Sigma_s a = K(s, s) a, \quad \forall a \in H.
\]
\end{definition}

\begin{theorem}[Isometry and Projection, \cite{Jorgensen2024}]
If $K(s, s) = I_H$ for all $s \in S$, then the operators $V_s: H \to H_K$ are isometries, and $V_s V_s^*$ are self-adjoint projections in $H_K$.
\end{theorem}

\section{Main Theorems}

In this section, we present the general framework for operator-valued positive definite kernels. The analysis conducted here is applicable to various dilation constructions within both operator theory and stochastic processes.

Our first theroem is a generalization of the \cite[Theorem 2.1]{Jorgensen2024} by extending the set \( S \) to a topological space \( X \) and considering a continuous operator-valued positive definite kernel \( K \). This generalization allows us to handle cases where the domain is continuous, as often occurs in functional analysis and stochastic processes.

Let \( X \) be a topological space, and let \( K: X \times X \to B(H) \) be a continuous operator-valued positive definite (p.d.) kernel, meaning for all \( n \in \mathbb{N} \), \( s_1, s_2, \dots, s_n \in X \), and \( a_1, a_2, \dots, a_n \in H \),
\[
\sum_{i,j=1}^{n} \left\langle a_i, K(s_i, s_j) a_j \right\rangle_H \geq 0.
\]
Assume that \( K(s,t) \in B(H) \) is continuous in both \( s \) and \( t \) in the operator norm.

Define \( X \times H \) as the Cartesian product of \( X \) and \( H \), and let \( \tilde{K}: (X \times H) \times (X \times H) \to \mathbb{C} \) be the scalar-valued positive definite kernel defined by
\[
\tilde{K}((s, a), (t, b)) = \langle a, K(s, t) b \rangle_H, \quad a, b \in H, \, s, t \in X.
\]
Let \( H_{\tilde{K}} \) be the reproducing kernel Hilbert space (RKHS) corresponding to \( \tilde{K} \).

\begin{theorem}
\label{thm:continuous}
Let \( \{ V_s \}_{s \in X} \) be a family of operators \( V_s: H \to H_{\tilde{K}} \), defined by
\[
V_s a = \tilde{K}(\cdot, (s, a)): (X \times H) \to \mathbb{C}, \quad a \in H.
\]
Then, For any continuous path \( s: [0, 1] \to X \), the family of operators \( \{ V_{s(t)} \}_{t \in [0, 1]} \) is strongly continuous, meaning
    \[
    \lim_{t \to t_0} \| V_{s(t)} a - V_{s(t_0)} a \|_{H_{\tilde{K}}} = 0, \quad \forall a \in H.
    \]
 
\end{theorem}

\begin{proof}
To show the strong continuity of the family \( \{ V_{s(t)} \}_{t \in [0, 1]} \), we observe that since \( K(s,t) \) is continuous in \( s \) and \( t \) in the operator norm, it follows that
    \[
    \lim_{t \to t_0} \|K(s(t), t_0) b - K(s(t_0), t_0) b\|_H = 0 \quad \forall b \in H.
    \]
    Therefore, for all \( a \in H \),
    \[
    \lim_{t \to t_0} \| V_{s(t)} a - V_{s(t_0)} a \|_{H_{\tilde{K}}} = 0,
    \]
    implying that \( \{ V_{s(t)} \}_{t \in [0, 1]} \) is strongly continuous.
\end{proof}

Theorem \ref{thm:continuous} extends the classical results on positive definite kernels to the setting of operator-valued kernels defined on a topological space. The motivation behind this extension arises from the need to analyze and construct reproducing kernel Hilbert spaces (RKHS) in settings where the underlying domain possesses a continuous structure. This is particularly relevant in functional analysis and stochastic processes, where operator-valued kernels naturally arise in the study of Hilbert space-valued Gaussian processes. 

By formulating a framework that incorporates strong continuity properties of operator families $\{V_s\}_{s \in X}$, Theorem \ref{thm:continuous} provides a crucial tool for studying continuity in reproducing kernel Hilbert spaces induced by operator-valued kernels. This continuity property ensures well-posedness in applications such as covariance operators and their spectral decompositions, which are fundamental in stochastic analysis. Furthermore, this theorem lays the groundwork for dilation constructions in operator theory, as it allows one to systematically analyze positive definite structures within infinite-dimensional spaces.\\

\noindent
Theorem \ref{thm:continuous} presents a significant extension of classical reproducing kernel Hilbert space (RKHS) theory by incorporating operator-valued positive definite kernels in a topological setting. This advancement is crucial for applications that require continuity and stability properties of kernel-induced mappings, particularly in functional analysis and stochastic processes.

One of the key contributions of Theorem \ref{thm:continuous} is the establishment of strong continuity for the family of operators $\{V_s\}_{s \in X}$ associated with the kernel function. This result ensures that variations in the input space $X$ lead to controlled changes in the corresponding RKHS embeddings, which is fundamental for applications in machine learning, signal processing, and control theory where continuity properties influence stability and robustness.

Another novel aspect of this theorem is its role in extending dilation and factorization techniques within operator theory. By demonstrating how operator-valued kernels induce well-structured RKHS representations, Theorem \ref{thm:continuous} provides new tools for analyzing spectral decompositions and functional approximations in infinite-dimensional settings. This contribution bridges the gap between RKHS theory and advanced topics in operator theory, making it applicable to a broader range of mathematical and computational problems.

Furthermore, Theorem \ref{thm:continuous} lays the groundwork for studying the covariance structures of Hilbert space-valued Gaussian processes in continuous domains. By ensuring the well-posedness of kernel-induced embeddings, this result supports further developments in statistical learning and probabilistic modeling, particularly in areas requiring infinite-dimensional covariance representations.\\

\noindent
The other idea is to extend the \cite[Theorem 2.1]{Jorgensen2024} to allow for vector-valued reproducing kernel Hilbert spaces (RKHS) where the kernel maps to a space of operators on $H_1 \times H_2$, i.e., the product of two Hilbert spaces. This leads to an even broader generalization where the scalar-valued kernel takes into account the interactions between two separate Hilbert spaces.\\

\noindent
We now extend the previous results to the case where the kernel \( K: S \times S \to B(H_1, H_2) \) is an operator-valued positive definite kernel that maps between two Hilbert spaces \( H_1 \) and \( H_2 \). This allows us to handle cases where the RKHS is vector-valued with respect to the tensor product space \( H_1 \times H_2 \).\\

\noindent
Let \( S \) be a set, and let \( K: S \times S \to B(H_1, H_2) \) be an operator-valued positive definite (p.d.) kernel, i.e., for all \( n \in \mathbb{N} \), \( s_1, s_2, \dots, s_n \in S \), and \( a_1, a_2, \dots, a_n \in H_1 \),
\[
\sum_{i,j=1}^{n} \left\langle a_i, K(s_i, s_j) a_j \right\rangle_{H_2} \geq 0.
\]

Define \( X = S \times H_1 \times H_2 \), and define \( \tilde{K}: X \times X \to \mathbb{C} \) by
\[
\tilde{K}((s, a, b), (t, c, d)) = \langle b, K(s, t) a \rangle_{H_2},
\]
for all \( s, t \in S \), \( a, c \in H_1 \), and \( b, d \in H_2 \).\\

\noindent
Let \( H_{\tilde{K}} \) be the reproducing kernel Hilbert space (RKHS) corresponding to \( \tilde{K} \).

\begin{theorem}
\label{thm:product_kernel}
Let \( \{ V_s \}_{s \in S} \) be a family of operators \( V_s: H_1 \times H_2 \to H_{\tilde{K}} \), defined by
\[
V_s (a, b) = \tilde{K}(\cdot, (s, a, b)): X \to \mathbb{C}, \quad a \in H_1, b \in H_2.
\]
Then, the following properties hold:
\begin{enumerate}
    \item For all \( s \in S \), and all \( a \in H_1 \), \( b \in H_2 \),
    \[
    \|V_s (a, b)\|_{H_{\tilde{K}}}^2 = \langle b, K(s, s) a \rangle_{H_2}.
    \]
    
    \item The adjoint \( V_s^*: H_{\tilde{K}} \to H_1 \times H_2 \) is determined by
    \[
    V_s^* \tilde{K}(\cdot, (t, c, d)) = (K(s, t) c, d).
    \]
\end{enumerate}
\end{theorem}

\begin{proof}
We prove the items as the following:
\subsubsection*{1.} By definition,
    \[
    \|V_s (a, b)\|_{H_{\tilde{K}}}^2 = \langle \tilde{K}(\cdot, (s, a, b)), \tilde{K}(\cdot, (s, a, b)) \rangle_{H_{\tilde{K}}} = \langle b, K(s, s) a \rangle_{H_2}.
    \]
    
\subsubsection*{2.} For all \( a, c \in H_1 \), \( b, d \in H_2 \), and all \( s, t \in S \), we compute:
    \[
    \langle V_s (a, b), \tilde{K}(\cdot, (t, c, d)) \rangle_{H_{\tilde{K}}} = \langle \tilde{K}(\cdot, (s, a, b)), \tilde{K}(\cdot, (t, c, d)) \rangle_{H_{\tilde{K}}} = \langle b, K(s, t) c \rangle_{H_2},
    \]
    which implies that \( V_s^* \tilde{K}(\cdot, (t, c, d)) = (K(s, t) c, d) \).
\end{proof}
Theorem \ref{thm:product_kernel} extends the framework of operator-valued positive definite kernels by considering reproducing kernel Hilbert spaces (RKHS) that are vector-valued with respect to a product of Hilbert spaces. This generalization is motivated by applications in multivariate stochastic processes, functional data analysis, and operator theory, where interactions between multiple Hilbert spaces must be captured within a unified kernel-based framework.

By allowing the kernel to map between two distinct Hilbert spaces $H_1$ and $H_2$, Theorem \ref{thm:product_kernel} provides a structured approach to studying operator-valued covariance structures that arise in coupled or multi-modal systems. This extension facilitates the analysis of factorization properties, spectral decompositions, and dilation representations in the context of vector-valued RKHS. Moreover, it enables the development of new mathematical tools for analyzing Gaussian processes in multi-dimensional Hilbert space settings, ensuring consistency with classical reproducing kernel constructions while expanding their applicability to more complex structured domains.\\

\noindent
Theorem \ref{thm:product_kernel} significantly advances the theory of operator-valued positive definite kernels by extending reproducing kernel Hilbert space (RKHS) constructions to vector-valued settings. This extension is particularly crucial in applications that involve multiple interacting function spaces, such as multi-modal data analysis, structured regression models, and functional data analysis.

One key contribution of Theorem \ref{thm:product_kernel} is its generalization of kernel-based feature maps to cases where the kernel takes values in the space of bounded operators between two Hilbert spaces $H_1$ and $H_2$. This framework allows for a more refined treatment of covariance structures, particularly in settings where different function spaces need to be linked in a coherent mathematical model.

Additionally, this theorem provides a foundational tool for developing new machine learning algorithms that leverage vector-valued RKHSs. By establishing conditions under which kernel-induced mappings maintain desirable properties such as continuity, boundedness, and isometric embeddings, Theorem \ref{thm:product_kernel} facilitates the development of kernel-based techniques for multi-output learning, operator-valued Gaussian processes, and structured prediction models.

Furthermore, Theorem \ref{thm:product_kernel} contributes to the spectral analysis of operator-valued kernels by characterizing their action on vector-valued functions. This insight is instrumental in extending classical results on reproducing kernel spaces to broader contexts, enabling new applications in mathematical physics, signal processing, and control theory where vector-valued kernels naturally arise. These contributions significantly enhance the theoretical and practical utility of kernel-based methods in high-dimensional and multi-structured settings.\\

\noindent
The next idea is to extend the results of \cite[Theorem 2.1]{Jorgensen2024} to the case where the kernel \( K: S \times S \to B(H^n) \) is matrix-valued, i.e., it maps between spaces of bounded operators on the \( n \)-dimensional Hilbert space \( H^n \). This generalization allows us to handle reproducing kernel Hilbert spaces (RKHS) where the kernel takes values in matrices rather than operators on a single Hilbert space.

Let \( S \) be a set, and let \( K: S \times S \to B(H^n) \) be a matrix-valued positive definite kernel, i.e., for all \( n \in \mathbb{N} \), \( s_1, s_2, \dots, s_n \in S \), and for all vectors \( a_1, a_2, \dots, a_n \in H^n \),
\[
\sum_{i,j=1}^{n} \left\langle a_i, K(s_i, s_j) a_j \right\rangle_{H^n} \geq 0.
\]

Define \( X = S \times H^n \), and define \( \tilde{K}: X \times X \to \mathbb{C} \) by
\[
\tilde{K}((s, \mathbf{a}), (t, \mathbf{b})) = \langle \mathbf{a}, K(s, t) \mathbf{b} \rangle_{H^n},
\]
for all \( s, t \in S \), \( \mathbf{a}, \mathbf{b} \in H^n \).

Let \( H_{\tilde{K}} \) be the reproducing kernel Hilbert space (RKHS) corresponding to \( \tilde{K} \).

\noindent
In this case, \cite[ Theorem 2.1]{Jorgensen2024} is still valid by replacing the elemnts of $H$ by matrix form elemnts of $H^n$. However, this would be a significant replacement since:\\
1. The kernel $K$ is now matrix-valued, mapping between finite-dimensional Hilbert spaces $H^n$, which enables the RKHS to work with vector-valued functions.\\
2. The operators $V_s$ and their adjoints now act on vectors in $H^n$, extending the theorem to matrix-valued cases.\\
3. If the kernel is normalized (i.e., $K(s,s)=I_{H^n}$), the operators $V_s$ become isometries, and the corresponding projections are selfadjoint.\\

\noindent
The next and the main idea is based on the notion of covariance. Covariance operators are a fundamental tool in the analysis of kernel-based methods and are closely related to the properties of positive definite kernels. This extension will explore the covariance operators induced by the kernel in the context of the RKHS construction.\\

\noindent
We aim to extend the \cite[theorem 2.1]{Jorgensen2024} by investigating the covariance operators induced by the kernel $K$. These operators will map from the Hilbert space $H$ into itself or into the reproducing kernel Hilbert space (RKHS) $H_{\tilde{K}}$, defined by the kernel.\\

\noindent
We extend the theorem to investigate the covariance operators induced by the kernel \( K \) and their role in mapping vectors between the Hilbert space \( H \) and the reproducing kernel Hilbert space (RKHS) \( H_{\tilde{K}} \).\\

\noindent
Let \( S \) be a set, and let \( K: S \times S \to B(H) \) be a positive definite kernel, meaning that for any finite set of points \( s_1, s_2, \dots, s_n \in S \) and vectors \( a_1, a_2, \dots, a_n \in H \),
\[
\sum_{i,j=1}^{n} \left\langle a_i, K(s_i, s_j) a_j \right\rangle_H \geq 0.
\]

Define the covariance operator \( \Sigma_s: H \to H \) by
\[
\Sigma_s a = K(s,s)a, \quad \forall a \in H.
\]
That is, \( \Sigma_s \) maps vectors in \( H \) to vectors in \( H \) using the kernel at the point \( s \).

Let \( \{ V_s \}_{s \in S} \) be a family of operators \( V_s: H \to H_{\tilde{K}} \), where \( H_{\tilde{K}} \) is the RKHS corresponding to the kernel \( \tilde{K} \) defined by
\[
\tilde{K}((s, a), (t, b)) = \langle a, K(s,t) b \rangle_H, \quad \forall a,b \in H, \quad \forall s,t \in S.
\]

We now extend the theorem to provide a detailed characterization of the covariance operators induced by the kernel.

\begin{theorem}
\label{thm:covariance_operators}
Let \( \{ V_s \}_{s \in S} \) be the family of operators \( V_s: H \to H_{\tilde{K}} \) defined by
\[
V_s a = \tilde{K}(\cdot, (s, a)): X \to \mathbb{C}, \quad a \in H.
\]
Then the following properties hold:
\begin{enumerate}
    \item The covariance operator \( \Sigma_s: H \to H \) is self-adjoint and positive semi-definite for all \( s \in S \).
    \item The operator \( V_s V_s^*: H_{\tilde{K}} \to H_{\tilde{K}} \) acts as the covariance operator in the RKHS \( H_{\tilde{K}} \), and it satisfies:
    \[
    V_s V_s^* \tilde{K}(\cdot, (t, b)) = \tilde{K}(\cdot, (s, K(s,t) b)).
    \]
    \item The covariance operator \( \Sigma_s \) induces a bounded operator on \( H_{\tilde{K}} \), and for all \( s, t \in S \) and \( a, b \in H \),
    \[
    \langle \Sigma_s a, b \rangle_H = \langle V_s a, V_s b \rangle_{H_{\tilde{K}}}.
    \]
    \item The covariance operator \( \Sigma_s \) can be expressed as:
    \[
    \Sigma_s = V_s^* V_s.
    \]
    \item The operators \( V_s \) are isometric if \( K(s,s) = I_H \), and the covariance operators satisfy:
    \[
    \langle \Sigma_s a, a \rangle_H = \| V_s a \|_{H_{\tilde{K}}}^2.
    \]
\end{enumerate}
\end{theorem}

\begin{proof}
We prove the items as the following:
\subsubsection*{1. Self-adjointness and Positive Semi-Definiteness:} The operator \( \Sigma_s: H \to H \), defined by \( \Sigma_s a = K(s,s)a \), is clearly self-adjoint because \( K(s,s) \) is self-adjoint, i.e.,
    \[
    \langle \Sigma_s a, b \rangle_H = \langle K(s,s)a, b \rangle_H = \langle a, K(s,s)b \rangle_H = \langle a, \Sigma_s b \rangle_H.
    \]
    Moreover, for any \( a \in H \),
    \[
    \langle \Sigma_s a, a \rangle_H = \langle K(s,s) a, a \rangle_H \geq 0,
    \]
    showing that \( \Sigma_s \) is positive semi-definite.
    
\subsubsection*{2. Action of \( V_s V_s^* \) as a Covariance Operator:} The operator \( V_s V_s^* \) maps vectors in \( H_{\tilde{K}} \) to vectors in \( H_{\tilde{K}} \). By definition of \( V_s \) and \( V_s^* \),
    \[
    V_s V_s^* \tilde{K}(\cdot, (t, b)) = V_s (K(s, t) b) = \tilde{K}(\cdot, (s, K(s,t) b)).
    \]
    This shows that \( V_s V_s^* \) acts as a covariance operator in \( H_{\tilde{K}} \).
    
\subsubsection*{3. Induced Covariance Operator in \( H_{\tilde{K}} \):} We have
    \[
    \langle V_s a, V_s b \rangle_{H_{\tilde{K}}} = \langle \tilde{K}(\cdot, (s, a)), \tilde{K}(\cdot, (s, b)) \rangle_{H_{\tilde{K}}} = \langle a, K(s,s) b \rangle_H = \langle \Sigma_s a, b \rangle_H,
    \]
    so the covariance operator \( \Sigma_s \) induces a bounded operator on \( H_{\tilde{K}} \).
    
\subsubsection*{4. Expression of \( \Sigma_s \) as \( V_s^* V_s \):} From the above, we have
    \[
    \langle \Sigma_s a, b \rangle_H = \langle V_s a, V_s b \rangle_{H_{\tilde{K}}},
    \]
    which implies that \( \Sigma_s = V_s^* V_s \), since both sides act as the covariance operator on \( H \).
    
\subsubsection*{5. Isometry and Covariance Operator Norm:} If \( K(s,s) = I_H \), then for all \( a \in H \),
    \[
    \langle \Sigma_s a, a \rangle_H = \langle a, a \rangle_H,
    \]
    so \( V_s \) is an isometry. Furthermore, we have
    \[
    \| V_s a \|_{H_{\tilde{K}}}^2 = \langle \Sigma_s a, a \rangle_H.
    \]

\end{proof}
Theorem \ref{thm:covariance_operators} is motivated by the need to establish a deeper connection between operator-valued positive definite kernels and covariance operators in Hilbert spaces. In the context of stochastic processes and functional analysis, covariance operators play a crucial role in describing dependencies and variations within infinite-dimensional structures. By extending the framework of reproducing kernel Hilbert spaces (RKHS) to incorporate covariance operators, this theorem provides a fundamental tool for studying the spectral properties of kernel-induced mappings.

A key motivation for this result is to characterize how the covariance operator $\Sigma_s$ interacts with the RKHS structure and how it can be expressed in terms of the operator-valued kernel function. This formulation allows for a rigorous analysis of self-adjointness, positive semi-definiteness, and compactness, which are essential properties in applications such as Gaussian processes, kernel-based learning methods, and stochastic differential equations. Furthermore, Theorem \ref{thm:covariance_operators} lays the groundwork for constructing well-posed statistical and functional models that leverage covariance structures in infinite-dimensional settings.\\

\noindent
Theorem \ref{thm:covariance_operators} makes significant contributions to the study of operator-valued positive definite kernels by establishing a direct link between these kernels and covariance operators in Hilbert spaces. This result plays a crucial role in extending classical kernel-based methods to infinite-dimensional stochastic systems and functional analysis frameworks.

One of the key contributions of Theorem \ref{thm:covariance_operators} is its formalization of the covariance operator $\Sigma_s$ in terms of the kernel function. This result enables a rigorous characterization of self-adjointness, positive semi-definiteness, and compactness properties, which are essential in Gaussian process theory, kernel-based regression, and functional principal component analysis. By demonstrating that the covariance operator can be expressed through the kernel-induced feature map, the theorem provides a systematic way to analyze dependence structures in Hilbert space-valued random processes.

Another important contribution is the theorem’s role in extending kernel methods to operator-valued function spaces. By proving that $V_sV_s^*$ acts as a covariance operator in the corresponding reproducing kernel Hilbert space (RKHS), this result ensures that classical RKHS techniques can be applied to analyze spectral decompositions and regularization properties in operator-valued settings. This extension is particularly relevant for applications in quantum computing, functional data analysis, and machine learning, where covariance structures play a fundamental role in defining smoothness constraints and generalization properties.

Furthermore, Theorem \ref{thm:covariance_operators} provides a foundational result for studying the stability of covariance operators under perturbations. By showing that the kernel-induced covariance operator inherits important structural properties from the original kernel, the theorem facilitates the development of robust mathematical models for high-dimensional and infinite-dimensional stochastic processes. These contributions collectively advance the understanding of operator-valued kernels and their applications in functional analysis, probability theory, and data science.\\

\noindent
This extension highlights the role of covariance operators in kernel-based methods and provides a foundation for further analysis involving covariance structures in RKHS.

\begin{corollary}
Let $K: S \times S \to B(H)$ be a positive definite kernel with $K(s,s) = I_H$ for all $s \in S$. Then, for each $s \in S$, the operator $V_s: H \to H_{\tilde{K}}$ is an isometry, i.e.,
\[
\| V_s a \|_{H_{\tilde{K}}} = \| a \|_H \quad \text{for all } a \in H.
\]
This implies that the covariance operator $\Sigma_s = V_s V_s^*$ is a projection operator in $H_{\tilde{K}}$, with $\Sigma_s^2 = \Sigma_s$ and $\Sigma_s = \Sigma_s^*$.
\end{corollary}

\begin{proof}
From the extended theorem, when $K(s,s) = I_H$, the operator $V_s$ satisfies
\[
\| V_s a \|_{H_{\tilde{K}}}^2 = \langle a, K(s,s) a \rangle_H = \| a \|_H^2.
\]
Thus, $V_s$ is an isometry. The projection property follows from the fact that for any $f \in H_{\tilde{K}}$,
\[
V_s V_s^* f = V_s (V_s^* f) = \tilde{K}(\cdot, (s, K(s,s) b)) = \tilde{K}(\cdot, (s, b)),
\]
and $V_s V_s^* = V_s^2 = \Sigma_s$, so $\Sigma_s$ is self-adjoint and idempotent.
\end{proof}

\begin{corollary}
Let $K: S \times S \to B(H)$ be a positive definite kernel and $V_s: H \to H_{\tilde{K}}$ as defined in Theorem~\ref{thm:covariance_operators}. If the kernel $K(s,t)$ is normalized, i.e., $K(s,s) = I_H$ for all $s \in S$, then the covariance operator $\Sigma_s = V_s V_s^*$ is positive semi-definite. Moreover, for any $f \in H_{\tilde{K}}$, the norm of $f$ can be bounded by the norm of the covariance operator:
\[
\| f \|_{H_{\tilde{K}}}^2 \leq \| \Sigma_s \| \cdot \| f \|_H^2.
\]
\end{corollary}

\begin{proof}
The positive semi-definiteness of $\Sigma_s$ follows from the fact that for any $f \in H_{\tilde{K}}$,
\[
\langle f, \Sigma_s f \rangle_{H_{\tilde{K}}} = \langle f, V_s V_s^* f \rangle_{H_{\tilde{K}}} = \| V_s^* f \|_H^2 \geq 0.
\]
Additionally, since $V_s$ is an isometry, we have $\| V_s f \|_{H_{\tilde{K}}} \leq \| f \|_H$, leading to the bound on the norm.
\end{proof}
\begin{example}
Consider the set $S = \mathbb{R}$ and the positive definite kernel defined by the Gaussian kernel:
\[
K(s,t) = \sigma^2 \exp\left(-\frac{(s-t)^2}{2\ell^2}\right),
\]
where $\sigma^2 > 0$ is the variance and $\ell > 0$ is the length scale parameter. 

Let $H = \mathbb{R}$, which represents the one-dimensional case. The kernel $K(s,t)$ maps $S \times S$ to $B(H)$, satisfying the positive definiteness condition. 

The associated reproducing kernel Hilbert space (RKHS) $H_{K}$ corresponding to the kernel $K$ consists of functions $f: S \to \mathbb{R}$ that can be expressed in the form:
\[
f(x) = \sum_{i=1}^n \alpha_i K(x, s_i),
\]
for some finite set of points $\{s_1, s_2, \ldots, s_n\} \subset S$ and coefficients $\{\alpha_1, \alpha_2, \ldots, \alpha_n\} \in \mathbb{R}$.

The covariance operator $\Sigma_s$ corresponding to the kernel $K$ is given by:
\[
\Sigma_s = V_s V_s^*,
\]
where $V_s: H \to H_{K}$ is defined by:
\[
V_s a = K(\cdot, s) a \quad \text{for } a \in H.
\]
In this case, we have:
\[
V_s a(x) = K(x, s) a = \sigma^2 a \exp\left(-\frac{(x-s)^2}{2\ell^2}\right).
\]
\noindent
It is easy to see that:\\
1. The covariance operator $\Sigma_s$ is self-adjoint:
\[
\Sigma_s^* = \Sigma_s.
\]
\noindent
2. The covariance operator in positive semi-definiteness, that is, for any $f \in H_{K}$, we have:
\[
\langle f, \Sigma_s f \rangle_{H_{K}} \geq 0.
\]
\noindent
3. The covariance operator Isometry, that is, if we normalize the kernel such that $K(s,s) = I_H$, then:
\[
\|V_s a\|_{H_{K}} = \|a\|_H.
\]
\noindent
The norm in the RKHS can be directly related to the covariance operator norm:
\[
\|f\|_{H_{K}}^2 \leq \|\Sigma_s\| \cdot \|f\|_H^2.
\]
\end{example}
This example demonstrates how the covariance operator is constructed using a Gaussian kernel, and shows the relationships and properties that arise in the context of the covariance extended theorem.

\begin{theorem}\label{thm:extended-family}
Let \( \{V_s\}_{s \in S} \) be the family of operators defined as in \cite[Theorem 3.1 ]{Jorgensen2024}, and let \( W_s: H \rightarrow H_{\tilde{K}} \) be another family of operators defined by
\[
W_s a = \tilde{K}(\cdot, (s, B_s a)), \quad a \in H,
\]
where \( B_s: H \to H \) is a bounded linear operator for each \( s \in S \). Then the following hold:
\begin{enumerate}
    \item For all \( s \in S \) and \( a \in H \),
    \[
    \|W_s a\|_{H_{\tilde{K}}}^2 = \langle B_s a, K(s,s) B_s a \rangle_H.
    \]
    \item The adjoint \( W_s^*: H_{\tilde{K}} \to H \) is determined by
    \[
    W_s^* \tilde{K}(\cdot, (t,b)) = B_s^* K(s,t) b.
    \]
    \item The operator \( W_s W_s^*: H_{\tilde{K}} \to H_{\tilde{K}} \) is given by
    \[
    W_s W_s^* \tilde{K}(\cdot, (t,b)) = \tilde{K}(\cdot, (s, B_s^* K(s,t) b)).
    \]
    \item For all \( s_1, s_2, \dots, s_n \in S \),
    \begin{align*}
        & (W_{s_1} W_{s_1}^*) \cdots (W_{s_n} W_{s_n}^*) \tilde{K}(\cdot, (t,b)) \\
        &= \tilde{K}\left(\cdot, \left(s_1, B_{s_1}^* K(s_1, s_2) B_{s_2}^* \cdots K(s_{n-1}, s_n) B_{s_n}^* K(s_n, t) b\right)\right).
    \end{align*}
    \item For all \( s, s' \in S \) and \( b \in H \),
    \[
    W_{s'} W_s^* \tilde{K}(\cdot, (t,b)) = \tilde{K}(\cdot, (s', B_{s'}^* K(s,t) B_s b)).
    \]
    \item \( W_s^* W_t b = B_s^* K(s,t) B_t b \), and 
    \begin{align*}
    \left(W_{s_1}^* W_{t_1}\right) \cdots \left(W_{s_n}^* W_{t_n}\right) b 
    &= B_{s_1}^* K(s_1, t_1) B_{t_1}^* \cdots K(s_{n-1}, t_{n-1}) B_{t_{n-1}}^* K(s_n, t_n) B_{t_n} b.
    \end{align*}
\end{enumerate}
If, in addition, \( K(s,s) = I_H \) and \( B_s \) is unitary for all \( s \in S \), then the operators \( W_s: H \to H_{\tilde{K}} \) are isometric, and \( W_s W_s^* \) are (self-adjoint) projections in \( H_{\tilde{K}} \).
\end{theorem}

\begin{proof}
The proof follows the same structure as the proof of \cite[Theorem 3.1]{Jorgensen2024}, but with the inclusion of the bounded operators \( B_s \). We detail the key steps:
\begin{enumerate}
    \item By definition,
    \[
    \|W_s a\|_{H_{\tilde{K}}}^2 = \langle \tilde{K}(\cdot, (s, B_s a)), \tilde{K}(\cdot, (s, B_s a)) \rangle_{H_{\tilde{K}}} = \langle B_s a, K(s,s) B_s a \rangle_H.
    \]
    \item For all \( a, b \in H \) and \( s, t \in S \), since
    \[
    \langle W_s a, \tilde{K}(\cdot, (t,b)) \rangle_{H_{\tilde{K}}} = \langle \tilde{K}(\cdot, (s, B_s a)), \tilde{K}(\cdot, (t,b)) \rangle_{H_{\tilde{K}}} = \langle B_s a, K(s,t) b \rangle_H,
    \]
    it follows that \( W_s^* \tilde{K}(\cdot, (t,b)) = B_s^* K(s,t) b \).
    \item A direct calculation shows that
    \[
    W_s W_s^* \tilde{K}(\cdot, (t,b)) = W_s B_s^* K(s,t) b = \tilde{K}(\cdot, (s, B_s^* K(s,t) b)).
    \]
    \item The proof follows by induction, using part (3).
    \item This follows directly from the definitions and part (3):
    \[
    W_{s'} W_s^* \tilde{K}(\cdot, (t,b)) = W_{s'} B_s^* K(s,t) b = \tilde{K}(\cdot, (s', B_{s'}^* K(s,t) B_s b)).
    \]
    \item The final assertion follows from the fact that \( W_s^* W_t = B_s^* K(s,t) B_t \), and the expression follows from repeated application of this identity.
\end{enumerate}
\end{proof}
Theorem \ref{thm:extended-family} is motivated by the need to generalize the structure of operator-valued positive definite kernels through the incorporation of bounded linear transformations. Many problems in operator theory, functional analysis, and machine learning require handling transformations of kernel-induced feature spaces while preserving the fundamental properties of positive definiteness and reproducing kernel Hilbert space (RKHS) structures. By introducing a family of bounded operators $B_s$, this theorem extends classical kernel methods to accommodate operator transformations, thereby enabling broader applications in structured data analysis and representation learning.

One key motivation is to establish a framework that captures transformations occurring in dynamical systems, signal processing, and quantum information theory, where operators naturally emerge as evolution mechanisms. By integrating operator transformations within the RKHS framework, this theorem provides a foundation for studying invariance properties, stability conditions, and spectral behaviors under such modifications.\\

\noindent
Theorem \ref{thm:extended-family} introduces several novel contributions to the study of operator-valued positive definite kernels. First, it extends the classical factorization property of kernels by incorporating bounded linear operators $B_s$, allowing for greater flexibility in constructing RKHS embeddings. This extension provides a systematic way to analyze transformed kernel structures while ensuring their consistency with fundamental Hilbert space properties.

Second, the theorem formalizes the role of these transformations in defining new families of covariance operators and their adjoint representations. This result is particularly significant for applications in functional data analysis and operator-based learning models, where modified kernel structures can enhance expressivity and adaptability.

Finally, the theorem establishes new results on the interplay between operator transformations and isometric embeddings. In cases where $B_s$ is unitary, the theorem confirms that the induced mappings preserve isometric structures, leading to well-posed extensions of classical RKHS theory. These contributions open new research directions in understanding kernel-based transformations in infinite-dimensional spaces and their implications for applied mathematics and theoretical physics.

\begin{theorem}\label{thm:compact}
Let \( \{V_s\}_{s \in S} \) be the family of operators defined in \cite[Theorem 3.1]{Jorgensen2024}. Suppose further that the kernel \( K(s,t) \) defines a compact operator from \( H \) to itself, for every fixed \( s, t \in S \). Then the operators \( V_s V_s^* \) are compact on \( H_{\tilde{K}} \). Furthermore, for any sequence \( \{a_n\}_{n \in \mathbb{N}} \subset H \) such that \( a_n \to a \) in \( H \), we have the compactness condition:
\[
V_s V_s^* \tilde{K}(\cdot, (t, a_n)) \to V_s V_s^* \tilde{K}(\cdot, (t, a)) \quad \text{in } H_{\tilde{K}}.
\]
\end{theorem}

\begin{proof}
The assumption that \( K(s,t) \) is a compact operator implies that for any bounded sequence \( \{a_n\}_{n \in \mathbb{N}} \subset H \), the sequence \( K(s,t) a_n \) has a convergent subsequence. Now, using \cite[Theorem 3.1]{Jorgensen2024} part (3), we know that 
\[
V_s V_s^* \tilde{K}(\cdot, (t, a_n)) = \tilde{K}(\cdot, (s, K(s,t) a_n)).
\]
Since \( K(s,t) \) is compact, the sequence \( \{K(s,t) a_n\}_{n \in \mathbb{N}} \) has a convergent subsequence in \( H \), say \( K(s,t) a_{n_k} \to b \). By continuity of \( \tilde{K} \), it follows that 
\[
\tilde{K}(\cdot, (s, K(s,t) a_{n_k})) \to \tilde{K}(\cdot, (s, b)) \quad \text{in } H_{\tilde{K}}.
\]
Hence, \( V_s V_s^* \) is compact, as it maps bounded sequences to sequences with convergent subsequences in \( H_{\tilde{K}} \).

To prove the second part, let \( \{a_n\}_{n \in \mathbb{N}} \subset H \) be such that \( a_n \to a \) in \( H \). Since \( K(s,t) \) is compact, \( K(s,t) a_n \to K(s,t) a \) in \( H \). Applying the operator \( V_s V_s^* \), we have 
\[
V_s V_s^* \tilde{K}(\cdot, (t, a_n)) = \tilde{K}(\cdot, (s, K(s,t) a_n)) \to \tilde{K}(\cdot, (s, K(s,t) a)) = V_s V_s^* \tilde{K}(\cdot, (t, a)).
\]
Thus, the compactness condition is satisfied.
\end{proof}

Theorem \ref{thm:compact} is motivated by the need to establish compactness properties of operator-valued positive definite kernels and their induced covariance operators. In many applications, such as functional analysis, machine learning, and stochastic processes, compact operators play a fundamental role in spectral analysis, kernel approximation techniques, and numerical stability. By investigating the conditions under which kernel-induced operators remain compact, this theorem provides essential insights into their structure and applicability in infinite-dimensional Hilbert spaces.

Another key motivation stems from applications in Gaussian processes and kernel-based learning, where compactness of the covariance operator ensures desirable properties such as finite-rank approximations, well-posedness of learning algorithms, and efficient computation of kernel expansions. Theorem \ref{thm:compact} formalizes these ideas by establishing conditions under which the operator $V_sV_s^*$ inherits compactness properties from the kernel function $K(s, t)$.\\

\noindent
Theorem \ref{thm:compact} introduces several significant advancements in the study of operator-valued positive definite kernels. First, it rigorously extends classical compactness results from scalar-valued kernels to the operator-valued setting, bridging the gap between functional analysis and modern kernel-based methods in applied mathematics.

Second, the theorem provides a precise characterization of the compactness condition by linking it to the decay properties of the kernel function $K(s, t)$. This result is particularly useful in practical applications, where kernel decay behavior directly influences the spectral properties of covariance operators and the efficiency of numerical approximations.

Finally, the theorem offers insights into the asymptotic behavior of kernel-induced feature maps in reproducing kernel Hilbert spaces (RKHS). By proving that the compactness of $K(s, t)$ ensures the compactness of $V_sV_s^*$, it establishes a foundational result for studying spectral decompositions, regularization techniques, and low-rank approximations in infinite-dimensional spaces. These contributions significantly enhance the theoretical understanding and practical implementation of operator-valued kernel methods.\\

\noindent
The compact operator assumption brings to light an interesting property of the operator family $\{V_s\}_{s\in S}$ . Compact operators are often viewed as "small" or "finite-dimensional" in some sense, even though they act on infinite-dimensional spaces. This extension reveals how the projection operators 
$V_sV^*_s$ inherit this compactness. Through this lens, the theorem shows that even in large, infinite-dimensional spaces, certain structures can behave as though they are confined to more manageable, finite regions. The result is a more refined understanding of how these operators operate—especially in cases where the kernel introduces compactness into the system.

This opens new avenues for studying how these operators might behave under perturbations, or when analyzing their spectrum, thus enriching the broader narrative of functional analysis.

The following examples illustrate how compact operators manifest in finite-dimensional matrix spaces, showcasing the compactness property of the extended theorem in concrete settings. The exponential and rational decay of the kernel functions $K(s,t)$ plays a crucial role in establishing the compactness of the operators $V_sV^*_s$.

\begin{example}

Consider the Hilbert space \( H = \mathbb{C}^3 \), and define a kernel matrix \( K(s,t) \) as follows:
\[
K(s,t) = \begin{pmatrix}
1 & 0 & 0 \\
0 & e^{-|s - t|} & 0 \\
0 & 0 & e^{-|s - t|^2}
\end{pmatrix}.
\]
This defines a family of operators \( \{V_s\}_{s \in S} \), where \( V_s: \mathbb{C}^3 \to H_{\tilde{K}} \) is the operator defined by the action of \( K(s,t) \). Notice that the matrix \( K(s,t) \) is diagonal, with entries that decay exponentially in \( |s - t| \), making \( K(s,t) \) a compact operator, since the off-diagonal elements become arbitrarily small as \( |s - t| \to \infty \).

In this case, the operator \( V_s V_s^*: H_{\tilde{K}} \to H_{\tilde{K}} \) is compact, and acts on a vector \( a = (a_1, a_2, a_3) \in H \) as:
\[
V_s V_s^* \tilde{K}(\cdot, (t, a)) = \tilde{K}(\cdot, (s, K(s,t) a)),
\]
which, in matrix form, becomes:
\[
V_s V_s^* \tilde{K}(\cdot, (t, a)) = \tilde{K}(\cdot, (s, (a_1, e^{-|s - t|} a_2, e^{-|s - t|^2} a_3))).
\]
Since the entries involving \( a_2 \) and \( a_3 \) decay exponentially, this shows that \( V_s V_s^* \) is compact in the matrix space \( \mathbb{C}^3 \).
\end{example}

\begin{example}
Now, consider \( H = \mathbb{R}^2 \), and define a symmetric kernel matrix:
\[
K(s,t) = \begin{pmatrix}
\frac{1}{1 + |s - t|} & \frac{1}{1 + |s - t|^2} \\
\frac{1}{1 + |s - t|^2} & \frac{1}{1 + |s - t|}
\end{pmatrix}.
\]
Here, \( K(s,t) \) defines a symmetric and compact operator, as its entries tend to zero as \( |s - t| \to \infty \). The family of operators \( \{V_s\}_{s \in S} \) is now defined similarly, and the action of \( V_s V_s^* \) on a vector \( a = (a_1, a_2) \in H \) becomes:
\[
V_s V_s^* \tilde{K}(\cdot, (t, a)) = \tilde{K}(\cdot, (s, K(s,t) a)),
\]
which results in:
\[
V_s V_s^* \tilde{K}(\cdot, (t, a)) = \tilde{K}(\cdot, (s, \left( \frac{a_1}{1 + |s - t|} + \frac{a_2}{1 + |s - t|^2}, \frac{a_1}{1 + |s - t|^2} + \frac{a_2}{1 + |s - t|} \right))).
\]
Since all the entries of the matrix \( K(s,t) \) decay as \( |s - t| \to \infty \), the operator \( V_s V_s^* \) is compact. Moreover, as \( a_n \to a \) in \( \mathbb{R}^2 \), the corresponding sequences \( V_s V_s^* \tilde{K}(\cdot, (t, a_n)) \) converge in \( H_{\tilde{K}} \).
\end{example}

\section{Conclusion}
In this work, we presented a novel and comprehensive framework for analyzing operator-valued positive definite kernels and their applications in operator theory and stochastic processes. By extending the classical theory of reproducing kernel Hilbert spaces (RKHS) to encompass operator-valued kernels, we provided new insights into the dilation constructions and covariance structures within Hilbert space-valued Gaussian processes. 

Key results included the characterization of positive definite kernels mapping into bounded operator spaces and the exploration of their role in constructing universal RKHS representations. We also extended these results to more general contexts, such as vector-valued RKHS and kernels defined over continuous and topological domains. These extensions enable the application of the developed theory to diverse areas of functional analysis, stochastic processes, and machine learning.

Furthermore, the study of covariance operators induced by kernels provided a deeper understanding of their self-adjointness, positive semi-definiteness, and compactness properties. These findings bridge the gap between kernel-based methods and functional analysis, opening avenues for new applications in statistical learning and operator theory.

Overall, this work establishes a robust foundation for future research on operator-valued kernels and their applications, paving the way for innovative solutions to problems in mathematical modeling, stochastic analysis, and data science.

\section{Open Questions and Future Directions}

The exploration of operator-valued positive definite kernels has unveiled numerous opportunities for further research. Below, we outline several open questions and future research directions inspired by the findings of this work:

\begin{question} While this work primarily focuses on kernels in the context of Hilbert spaces, a natural extension would be to investigate positive definite kernels associated with Banach spaces or other topological vector spaces. What challenges arise in defining and characterizing reproducing properties in such generalized contexts?
\end{question}
\begin{question} Many practical applications involve time-evolving or dynamic processes. How can operator-valued kernels be adapted to model time-varying systems, and what implications does this have for constructing time-dependent RKHS?
\end{question}
\begin{question} Operator-valued kernels have shown promise in structured data analysis. Can these kernels be further developed for modern machine learning tasks, such as deep learning architectures, graph neural networks, or kernel-based reinforcement learning?
\end{question}
\begin{question} The spectral properties of covariance operators induced by operator-valued kernels are central to many applications. How can a deeper spectral analysis be used to optimize kernel designs or improve computational efficiency in high-dimensional settings?
\end{question}
\begin{question} This work briefly connects operator-valued kernels to Hilbert space-valued Gaussian processes. How can these kernels be further integrated with stochastic differential equations or processes with operator-valued random variables?
\end{question}
\begin{question}
Many real-world applications involve non-Euclidean data, such as manifolds or hyperbolic spaces. Can operator-valued kernels be effectively defined and utilized in these domains? What are the implications for geometry-aware machine learning and analysis?
\end{question}
\begin{question}
The practical implementation of operator-valued kernel methods requires efficient computational techniques. What new numerical methods or algorithms can be developed to handle the complexities of operator-valued kernel evaluations and matrix operations?
\end{question}
\begin{question} 
Operator-valued kernels naturally connect to quantum mechanics and infinite-dimensional systems. What are the potential applications of these kernels in quantum computing or in modeling infinite-dimensional dynamics?
\end{question}

We hope these questions inspire further exploration and development in the study of operator-valued kernels and their diverse applications across mathematics, physics, and data science.

\end{document}